\documentclass[11pt]{article}

\usepackage{amssymb,amsmath,amsthm,fullpage,graphicx}
\usepackage{hyperref}
\usepackage{authblk}
\usepackage{cite}
\usepackage{caption}
\usepackage{subcaption}

\newtheorem{theorem}{Theorem}[section]
\newtheorem{lemma}[theorem]{Lemma}
\newtheorem{corollary}[theorem]{Corollary}
\newtheorem{proposition}[theorem]{Proposition}
\newtheorem{question}[theorem]{Question}

\newtheorem{observation}[theorem]{Observation}

\theoremstyle{definition}
\newtheorem{definition}[theorem]{Definition}

\def\ZZ{\mathbb{Z}}

\title{Vertex-transitive CIS graphs\thanks{Authors' e-mail addresses: \texttt{dobson@math.msstate.edu} (Edward Dobson),
\texttt{ademir.hujdurovic@upr.si} (Ademir Hujdurovi\'c), \texttt{martin.milanic@upr.si} (Martin Milani\v c),
\texttt{gabriel.verret@uwa.edu.au} (Gabriel Verret).}}

\begin{document}

\author[1,2]{Edward Dobson}
\author[2,3]{Ademir Hujdurovi\'{c}}
\author[2,3,4]{Martin Milani\v c}
\author[2,3,5]{Gabriel Verret}

\affil[1]{\scriptsize{}Department of Mathematics and Statistics}
\affil[ ]{Mississippi State University, Mississippi State, MS 39762, USA}
\affil[ ]{}
\affil[2]{University of Primorska, UP IAM}
\affil[ ]{Muzejski trg 2, SI-6000 Koper, Slovenia}
\affil[ ]{}
\affil[3]{University of Primorska, UP FAMNIT} 
\affil[ ]{Glagolja\v{s}ka 8, SI-6000 Koper, Slovenia}
\affil[ ]{}
\affil[4]{IMFM, Jadranska 19, 1000 Ljubljana, Slovenia}
\affil[ ]{}
\affil[5]{School of Mathematics and Statistics}
\affil[ ]{University of Western Australia, 35 Stirling Highway, Crawley, WA 6009, Australia}

\maketitle

\pagestyle{plain}

\begin{abstract}
A {\it CIS} graph is a graph in which every maximal stable set and every maximal clique intersect. A graph is {\em well-covered} if all its maximal stable sets are of the same size, {\em co-well-covered} if its complement is well-covered, and {\it vertex-transitive} if, for every pair of vertices, there exists an automorphism of the graph mapping one to the other.  We show that a vertex-transitive graph is CIS if and only if it is well-covered, co-well-covered, and the product of its clique and stability numbers equals its order. A graph is {\it irreducible} if no two distinct vertices have the same neighborhood. We classify irreducible well-covered CIS graphs with clique number at most $3$ and vertex-transitive CIS graphs of valency at most $7$, which include an infinite family. We also exhibit an infinite family of vertex-transitive CIS graphs which are not Cayley.

\medskip


\medskip

\noindent \small {\bf Math.~Subj.~Class.}~(2010): 05C69, 05C25
\end{abstract}

\section{Introduction}\label{sec:intro}

A {\it CIS} graph is a graph in which every maximal stable set and every maximal clique intersect (CIS stands for ``Cliques Intersect Stable sets''). The study of CIS graphs is rooted in observations of Berge~\cite{Berge} and Grillet~\cite{MR0244112}  (see \cite{MR1344757}). CIS graphs were studied in a series of papers~\cite{AndBorGur,BorosGM2014,MR2657704,MR2489416,MR2496915,MR2064873,MR1344757,MR2145519,BorosGM2014+}. The problem of recognizing CIS graphs is believed to be co-NP-complete~\cite{MR1344757}, conjectured to be co-NP-complete~\cite{MR2234986}, and conjectured to be polynomial~\cite{AndBorGur}. For further background on CIS graphs, see, e.g.,~\cite{BorosGM2014}.

A graph is called {\it vertex-transitive} if, for every pair of vertices, there exists an automorphism of the graph mapping one to the other. Our goal in this paper is to study vertex-transitive CIS graphs.

Our first main result is that a vertex-transitive graph is CIS if and only if it is well-covered, co-well-covered, and the product of its clique and stability numbers equals its order (Theorem~\ref{thm:VT-CIS}). (A graph is {\em well-covered} if all its maximal stable sets are of the same size and {\em co-well-covered} if its complement is well-covered.) We then exhibit several infinite families of vertex-transitive CIS graphs (see Section~\ref{sec:examples}), including some non-Cayley ones (see Proposition~\ref{prop:Gamma_r}).

In view of Theorem~\ref{thm:VT-CIS}, we spend some time studying well-covered graphs. (Well-covered graphs were defined by Plummer~\cite{MR0289347} and are well studied in the literature; see, for example, the survey~\cite{MR1254158} and~\cite{MR2739910, MR2602814,MR2568844,MR2537505, MR2340625, MR2206373} for some more recent references.) In particular,  we give a full classification of  irreducible well-covered CIS graphs with clique number at most $3$ (Theorem~\ref{Omega=3}).

These results are then used to prove our main result: a classification of connected vertex-transitive CIS graphs of valency at most $7$ (Corollary~\ref{cor:VT-CIS-valency-7}). In particular, we show that there are only finitely many such graphs for valency at most $6$ but there exists an interesting infinite family of examples with valency $7$. In fact, we prove Corollary~\ref{cor:VT-CIS-valency-7} under an hypothesis slightly weaker than vertex-transitive (see Theorem~\ref{thm:valency-7}). We conclude the paper with a few open problems (Section~\ref{sec:questions}).

\section{Preliminaries}\label{sec:prelim}

All graphs considered are finite, simple and undirected. Let $\Gamma=(V,E)$ be a graph. We call $V$ the vertex set of $\Gamma$ and write $V= V(\Gamma)$. Similarly,  we call $E$ the edge set of $\Gamma$ and write $E = E(\Gamma)$. The {\it complement} ${\overline{\Gamma}}$  of $\Gamma$ is the graph with the same vertex set and the complementary edge set $\overline E = \{\{x,y\}\,\mid \,x,y\in V,~x\neq y \textrm{ and } \{x,y\}\not\in E\}$.

For a vertex $v\in V$, let $N(v)$ denote the {\it neighborhood} of $v$, that is, the set of vertices of $\Gamma$ that are adjacent to $v$. The {\it closed neighborhood} of $v$ is the set $N(v)\cup\{v\}$, denoted by $N[v]$. The {\it valency} of $v$ is $|N(v)|$ and $\Gamma$ is said to be {\it $k$-regular} (or we say that it has \emph{valency} $k$) if all its vertices have valency $k$. A {\it universal} vertex of $\Gamma$ is a vertex adjacent to all other vertices of $V(\Gamma)$. If $v$ is a non-isolated vertex of $\Gamma$ then the \emph{local graph of $\Gamma$ at $v$} is the subgraph of $\Gamma$ induced by $N(v)$.

A {\it clique} (respectively, a {\it stable set}) of $\Gamma$ is {a set} of pairwise adjacent (respectively, non-adjacent) vertices. The inclusion maximal cliques and stable sets of $\Gamma$  are called {\it maximal}. The maximal cardinality of a clique (respectively a stable set) of $\Gamma$ is called the \emph{clique} (respectively \emph{stability}) number and denoted $\omega(\Gamma)$ (respectively $\alpha(\Gamma)$). We say that $\Gamma$ is {\it triangle-free} if $\omega(\Gamma)\leq 2$.

A {\it matching} in $\Gamma$ is a set of pairwise disjoint edges. A matching is \emph{perfect} if every vertex is incident with some edge of the matching. The {\it line graph} of $\Gamma$, denoted by $L(\Gamma)$, is the graph with vertex set $E(\Gamma)$ and two edges of $\Gamma$ adjacent in $L(\Gamma)$ if and only if they intersect. For positive integers $m$ and $n$, we denote by $n\Gamma$ the disjoint union of $n$ copies of $\Gamma$, by $C_n$ the cycle of order $n$, by $K_n$ the complete graph of order $n$ and by $K_{m,n}$ the complete bipartite graph with parts of size $m$ and $n$. The {\it lexicographic product} of graphs $\Gamma_1$ and $\Gamma_2$ is the graph $\Gamma_1[\Gamma_2]$ with vertex set $V(\Gamma_1)\times V(\Gamma_2)$, where two vertices $(u,x)$ and $(v,y)$ are adjacent if and only if either $\{u,v\}\in E(\Gamma_1)$ or $u = v$ and $\{x,y\}\in E(\Gamma_2)$.

Let $G$ be a group and let $S$ be an inverse-closed subset of $G$ such that $1\not\in S$.  The {\it Cayley graph of $G$ with connection set $S$} has vertex set $G$ with two vertices $g$ and $h$ adjacent if and only if $g^{-1}h\in S$. We say that a graph is \emph{Cayley} if it isomorphic to some Cayley graph. It is well known that a graph is Cayley if and only if its automorphism group contains a subgroup acting regularly (that is, sharply transitively) on vertices (see, e.g.,~\cite{Sabidussi1958}). In particular, Cayley graphs are vertex-transitive.

\subsection{CIS graphs}\label{sec:prelim-CIS-well-covered}

We first recall some basic properties of CIS graphs (see, e.g.,~\cite{BorosGM2014}).

\begin{proposition}\label{lemma:CIS-complements}\mbox{}
\begin{enumerate}
\item A graph is CIS if and only if its complement is CIS.
\item A disconnected graph is CIS if and only if each of its connected component is CIS. \label{CIS:disconnected}
\item For every two graphs $\Gamma_1$ and $\Gamma_2$, the lexicographic product $\Gamma_1[\Gamma_2]$ is CIS if and only if $\Gamma_1$ and $\Gamma_2$ are CIS. \label{CIS:lexproduct}
\end{enumerate}
\end{proposition}

In view of Proposition~\ref{lemma:CIS-complements}~(\ref{CIS:disconnected}), we often restrict ourselves to the study of connected graphs.

\begin{definition}
Let $\Gamma$ be a graph and let $R$ be the equivalence relation ``having the same neighborhood'' on $V(\Gamma)$. If $R$ is the identity relation then we say that $\Gamma$ is \emph{irreducible}. It is not hard to see that the quotient graph of $\Gamma$ with respect to $R$ is irreducible. It is called the \emph{irreducible quotient} of $\Gamma$.
\end{definition}

It is easy to see that every equivalence class of $R$ is a stable set\footnote{The equivalence classes of $R$ have appeared in the literature under various names such as {\it maximal independent-set modules}~\cite{HeggernesMP2011}, or {\it similarity classes}~\cite{LozinM2010}.}, and thus, given an irreducible graph $X$ one can easily recover all graphs which have $X$ as an irreducible quotient.

\begin{proposition}\label{irreducibleQuotient}
A graph is CIS if and only if its irreducible quotient is CIS.
\end{proposition}

\begin{proof}
Let $\Gamma$ be a graph, let $X$ be its irreducible quotient and let $\pi$ be the natural projection from $\Gamma$ to $X$. It is not hard to see that a maximal clique of $\Gamma$ consists in choosing exactly one representative of each $\pi$-fiber of a maximal clique of $X$. On the other hand, a maximal stable set of $\Gamma$ consists in choosing all elements of each fiber of a maximal stable set of $X$. The result easily follows.
\end{proof}

By Proposition~\ref{irreducibleQuotient} and the comment preceding it, the study of CIS graphs is reduced to the study of irreducible CIS graphs. We now develop some properties of irreducible CIS graphs.

\begin{lemma}\label{coollemma}
Let $\Gamma$ be an irreducible CIS graph. If $\omega(\Gamma)=t$ then no two $t$-cliques of $\Gamma$ intersect in a $(t-1)$-clique.
\end{lemma}

\begin{proof}
Suppose, on the contrary, that $C_1$ and $C_2$ are two $t$-cliques of $\Gamma$ such that $|C_1\cap C_2|=t-1$. Let $v_1$ be the unique vertex contained in $C_1$ but not in $C_2$ and let $v_2$ be the unique vertex contained in $C_2$ but not in $C_1$.

Let $x$ be a neighbor of $v_1$. If $x$ is not adjacent to $v_2$ then $\{v_2,x\}$ is contained in some maximal stable set $S$ which must intersect the maximal clique $C_1$, but $x$ is adjacent to $v_1$ and $v_2$ is adjacent to every vertex of $C_1\setminus\{v_1\}$, which is a contradiction. It follows that $x$ is adjacent to $v_2$. We have shown that $N(v_1)\subseteq N(v_2)$ and, by symmetry, we obtain $N(v_1)= N(v_2)$, contradicting the fact that $\Gamma$ is irreducible.
\end{proof}

The following are immediate consequences.

\begin{corollary}\label{corollary:localCIS}
Let $\Gamma$ be an irreducible co-well-covered CIS graph with $\omega(\Gamma)=t$, let $v$ be a non-isolated vertex of $\Gamma$ and let $Y$ be the local graph of $\Gamma$ at $v$. Then $Y$ is a co-well-covered graph with $\omega(Y)=t-1$ and with the property that no two $(t-1)$-cliques intersect in a $(t-2)$-clique.
\end{corollary}

\begin{proof}
 Clearly, $C_Y$ is a clique of $Y$ if and only if $\{v\}\cup C_Y$ is a clique of $\Gamma$ containing $v$.  The result then follows from Lemma~\ref{coollemma}.
\end{proof}

\begin{corollary}\label{Omega=2}
Let $\Gamma$ be a connected irreducible graph with $\omega(\Gamma)=2$. Then $\Gamma$ is CIS if and only if $\Gamma\cong K_2$.
\end{corollary}

\begin{corollary}\label{YOYO}
Let $\Gamma$ be a connected graph with $\omega(\Gamma)=2$. Then $\Gamma$ is CIS if and only if $\Gamma\cong K_{n,m}$ for some $n,m\geq 1$.
\end{corollary}

We remark that Corollary~\ref{YOYO} can also be easily derived from the following easy observation (see, e.g.,~\cite{AndBorGur}).

\begin{observation}\label{obs:bull}
Let $\Gamma$ be a CIS graph. If  $P$ is an induced path of length three in $\Gamma$ then there exists a vertex in $\Gamma$ adjacent to the two midpoints of $P$ and non-adjacent to its endpoints.
\end{observation}

Our last result of this section shows that, among connected CIS graphs, the graphs $L(K_{n,n})$ are characterized by their local graphs. This will be needed in later sections.

\begin{theorem}\label{Local TwiceComple}
Let $\Gamma$ be a connected CIS graph such that every local graph of $\Gamma$ is the disjoint union of two complete graphs. Then $\Gamma\cong L(K_{n,n})$ for some $n\geq 1$.
\end{theorem}

\begin{proof}
If $\Gamma$ contains an isolated vertex then, by connectedness, we have $\Gamma\cong K_1\cong L(K_{1,1})$. We thus assume that $\Gamma$ has no isolated vertex. Since every local graph of $\Gamma$ consists of two disjoint complete graphs, every vertex is contained in exactly two maximal cliques, and these two cliques intersect in only that vertex. Let $u$ and $v$ be distinct vertices. By the previous observation, there exists a maximal clique containing $u$ but not $v$ and thus $N[u]\neq N[v]$.

It follows that distinct vertices have distinct closed neighborhoods and, by~\cite[Theorem~16]{KloksKM1995}, $\Gamma$ is the line graph of a triangle-free graph $Z$. Clearly, $Z$ is connected. A stable set in $\Gamma$ corresponds to a matching in $Z$ and a maximal stable set in $\Gamma$ corresponds to a maximal matching in $Z$. Since $Z$ is triangle-free, a clique in $\Gamma$ corresponds to a star in $Z$ (that is, to a set of edges incident with a fixed vertex), and a maximal clique in $\Gamma$ corresponds to an maximal star in $Z$ (that is, to the set of all edges incident with a fixed vertex).

Since $\Gamma$ is CIS, every maximal stable set of $\Gamma$ intersects every maximal clique of $\Gamma$. Consequently, every maximal matching of $Z$ intersects every maximal star of $Z$, which means that every maximal matching of $Z$ is actually a perfect matching. By~\cite[Theorem~1]{Sumner1979}, it follows that $Z$ is isomorphic to $K_{2n}$ for some $n\ge 2$ or to $K_{n,n}$ for some $n\ge 1$. Since $Z$ is triangle-free, $Z\cong K_{n,n}$ and the result follows.
\end{proof}

\section{Vertex-transitive CIS graphs}\label{sec:VT}

Our first important result is the following characterization of CIS vertex-transitive graphs, which generalizes~\cite[Theorem~3]{BorosGM2014}.

\begin{theorem}\label{thm:VT-CIS}
Let $\Gamma$ be a vertex-transitive graph of order $n$. Then $\Gamma$ is CIS if and only if all maximal stable sets are of size $\alpha(\Gamma)$, all maximal cliques are of size $\omega(\Gamma)$, and $\alpha(\Gamma)\omega(\Gamma)= n$.
\end{theorem}

\begin{proof}
Let $G$ be the automorphism group of $\Gamma$, let $C$ be a maximal clique, and let $S$ be a maximal stable set of $\Gamma$. Denote by $\mathcal{I}(C,S)$ the set of triples $(x,y,g)$ such that $x\in C$, $y\in S$, $g\in G$ and $x^g=y$. Given $(x,y)\in C\times S$, the set of elements of $G$ which map $x$ to $y$ is a coset of $G_x$ and hence has cardinality $|G_x|$. This shows that $|\mathcal{I}(C,S)|=|C||S||G_x|$.  For every $g\in G$, $C^g$ is a maximal clique and hence $|C^g\cap S|\leq 1$.

If $\Gamma$ is CIS then $|C^g\cap S|=1$ for every $g\in G$. In particular, for every $g\in G$, there is a unique choice of $(x,y)\in C\times S$ such that $(x,y,g)\in \mathcal{I}(C,S)$. It follows that $|G|=|\mathcal{I}(C,S)|$.  By the orbit-stabilizer theorem, we have $|G|=n|G_x|$ and hence $n=|C||S|$. Since the choice of $C$ and $S$ was arbitrary, all maximal stable sets are of size $|S|$ and all maximal cliques are of size $|C|$.

Conversely, if $|C||S|=n$ then $|\mathcal{I}(C,S)|=|C||S||G_x|=n|G_x|=|G|$. Since $|C^g\cap S|\leq 1$ for every $g\in G$, it follows that in fact $|C^g\cap S|=1$ for every $g\in G$. Setting $g=1$, we obtain $|C\cap S|=1$. Since $C$ and $S$ were arbitrary, it follows that $\Gamma$ is CIS.
\end{proof}

\begin{corollary}\label{cor:regular-wc-co-wc}
Every vertex-transitive CIS graph is regular, well-covered, and co-well-covered.
\end{corollary}

Corollary~\ref{cor:regular-wc-co-wc} suggests that, in order to study vertex-transitive CIS graphs, it can be useful to first study well-covered CIS graphs. This is what we do in Section~\ref{sec:omega-3} but, first, we construct some infinite families of connected irreducible vertex-transitive CIS graphs.

\subsection{Examples}\label{sec:examples}

The proof of the following result is straightforward.

\begin{proposition}\label{prop:LKnn}
If $n\geq 1$ then $L(K_{n,n})$ is a connected vertex-transitive CIS graph of order $n^2$ and valency $2(n-1)$ with $\alpha(L(K_{n,n}))=\omega(L(K_{n,n}))=n$.
\end{proposition}

\begin{definition}
For $n\geq 3$, let $PX(n)$ be the graph with vertex-set $\ZZ_n\times\ZZ_2\times\ZZ_2$ and edge-set $\{(i,x,y),(i + 1,y,z)\mid i\in \ZZ_n,~x,y,z\in\ZZ_2\}$. To $PX(n)$, we add the following set of edges $\{(i,x,y),(i,u,v)\mid i\in \ZZ_n,~x,y,u,v\in\ZZ_2,~(x,y)\neq (u,v)\}$  to obtain the graph $Q_n$.
\end{definition}

The graphs $PX(n)$ were first studied in~\cite{PraegerX1989} (where they were denoted $C(2,n,2)$). Proposition~\ref{prop:Gamma_r} shows that the graphs $Q_n$ are vertex-transitive CIS graphs of valency $7$.  It turns out that there are only finitely many other vertex-transitive CIS graphs of valency at most $7$ (see Corollary~\ref{cor:VT-CIS-valency-7}). Proposition~\ref{prop:Gamma_r} also shows that $Q_n$ is not Cayley when $n$ is prime.

\begin{proposition}\label{prop:Gamma_r}
Let $n\geq 4$. The graph $Q_n$ is a connected vertex-transitive CIS graph of order $4n$ and valency $7$ with  $\alpha(Q_n)=n$ and $\omega(Q_n)=4$. Moreover, if $n$ is prime then $Q_n$ is not a Cayley graph.
\end{proposition}

\begin{proof}
Clearly, $Q_n$ has order $4n$, valency $7$, and is  connected and irreducible. Let $C$ be a clique of $Q_n$. Since vertices that are adjacent in $Q_n$ have first coordinate differing by at most one, it follows that $C$ is contained in the graph induced on vertices having first coordinate either $i$ or $i+1$, for some $i\in \ZZ_n$. It is easy to check that every maximal clique of this graph (which is isomorphic to $\overline{2C_4}$) has size $4$. It follows that $Q_n$ is co-well-covered with $\omega(Q_n)=4$.

Let $S$ be a maximal stable set of $Q_n$.  Clearly, for every $i\in\ZZ_n$,  $S$ contains at most one vertex with first coordinate $i$.  Suppose that, for some $i\in\ZZ_n$, $S$ does not contain a vertex with first coordinate $i$. Without loss of generality, we may assume that $(i-1,a,b),(i+1,c,d)\in S$ for some $a,b,c,d\in \ZZ_2$. Note that $(i,b+1,c+1)$ is adjacent to neither $(i-1,a,b)$ nor $(i+1,c,d)$ and thus to no element of $S$. This contradicts the maximality of $S$. It follows that $Q_n$ is well-covered with $\alpha(Q_n)=n$.

We say that two vertices of $Q_n$ are \emph{related} if they have the same first coordinate. This is clearly an equivalence relation. We denote the corresponding partition of $V(Q_n)$ by $\mathcal{B}$. If $n\geq 5$, it follows by \cite[Lemma 2.8]{PraegerX1989} that every automorphism of $PX(n)$ preserves $\mathcal{B}$. Since the graph obtained from $PX(n)$ by adding an edge between every related pair of vertices is $Q_n$, it follows that every automorphism of $PX(n)$ is an automorphism of $Q_n$. By~\cite[Theorem 2.10]{PraegerX1989}, $PX(n)$ is vertex-transitive and thus so is $Q_n$. If $n=4$ then, by~\cite[Theorem 2.10]{PraegerX1989}, $PX(n)$ contains a vertex-transitive group of automorphisms that preserves $\mathcal{B}$ and we argue as before. We have shown that $Q_n$ is vertex-transitive and thus by Theorem \ref{thm:VT-CIS}, it is CIS.

Call an edge of $Q_n$ \emph{red} if it is contained in two $4$-cliques. It is easy to check that two vertices of $Q_n$ are connected by a red path if and only if they have the same first coordinate. Since red edges are mapped to red edges by automorphisms of $Q_n$, this shows that $\mathcal{B}$ is preserved under the group of automorphism of $Q_n$. On the other hand, if we remove all edges between related vertices, we obtain $PX(n)$ and thus every automorphism of $Q_n$ is also an automorphism of $PX(n)$. Together with the previous paragraph, this implies that if $n\geq 5$ then $Q_n$ and $PX(n)$ have the same automorphism group.

Recall that a graph is Cayley if and only its automorphism group contains a regular subgroup and thus $Q_n$ is Cayley if and only if $PX(n)$ is. On other hand, it was shown in~\cite[Theorem 3]{MR1250993} that $PX(n)$ is not Cayley when $n$ is prime. This concludes the proof.
\end{proof}

\begin{proposition}\label{prop:R_n}
Let $n\geq 2$ and let $R_n$ be the Cayley graph on  $\ZZ_{2n}\times \ZZ_4$ with connection set {$S=\{(0,1),(0,3),(n,0),(n,2),(2i,2),(2i+1,0)\mid 0\leq i\leq n-1\}$.} Then $R_n$ is a connected CIS graph of order $8n$ and valency $2n+3$ with $\alpha(R_n)=2n$ and $\omega(R_n)=4$.
\end{proposition}

\begin{proof}
Clearly $R_n$ is of valency $2n+3$ and, since $S$ generates $\ZZ_{2n}\times \ZZ_4$, $R_n$ is connected.  Since $R_n$ is vertex-transitive, it suffices by Theorem~\ref{thm:VT-CIS} to prove that all maximal cliques containing $(0,0)$ are of size $4$ and that all maximal stable sets containing $(0,0)$ are of size $2n$. We will assume that $n$ is even. The argument in the case when $n$ is odd is analogous.

It is not difficult to see that every maximal clique containing $(0,0)$ is one of the following:  $\{(0,0),(0,1),(0,2),(0,3)\}$, $\{(0,0),(2i+1,0),(n+2i+1,0),(n,0)\}$  or $\{(0,0),(2i,2),(n+2i,2),(n,0)\}$,  for some $i\in \{0,1,\ldots,n-1\}$. Therefore, all maximal cliques containing $(0,0)$ are of size $4$.

Let $M$ be a maximal stable set containing $(0,0)$. Note that $M\cap S=\emptyset$. Suppose first that $M$ is contained in  $\ZZ_{2n}\times \{0,2\}$. Since $(0,0)\in M$ and $M$ is a stable set, $M$ is contained in $\bigcup_{i\in \{0,1,\ldots, n-1\}}\{(2i,0),(2i+1,2)\}$.  If $(2j,0)\in M$ for some $j\in \{1\ldots,n-1\}$ then $(2j+ n,1)$ has no neighbor in $M$, contradicting the fact that $M$ is a maximal stable set. Similarly, if $(2j+1,2)\in M$ for some $j\in \{0,1\ldots,n-1\}$ then $(n+2j+1,1)$ has no neighbor in $M$, again a contradiction.

We may thus assume that $M$ contains some element $(x,y)$ with $y\in\{1,3\}$. We will deal with the case when $y=1$ and $x=2k-1$ is odd. The other cases can be dealt with similarly. Let $X=\ZZ_{2n}\times \ZZ_4\setminus (N[(0,0)] \cup N[(2k-1,1)])$. Note that $M\setminus \{(0,0),(2k-1,1)\}\subseteq X$ and that
$$X=
\left(\bigcup_{\substack{i\in \{1,3,\ldots,2n-1\}\setminus\{2k-1\}\\j\in \{2,4,\ldots,2n-2\}}}
\{(j,0),(i,1),(i,2),(j,3)\}\right)\setminus \{(n,0),(n+2k-1,1)\}
$$
Thus, the subgraph of $R_n$ induced by $X$ is the disjoint union of two isolated vertices ($(n,3)$ and $(n+2k-1,2)$) and  $(n-2)$ cycles of length $4$ (of the form $((i,1),(i,2),(n+i,2),(n+i,1))$ for $i\in \{1,3,\ldots,2n-1\}\setminus\{2k-1\}$ or $((j,0),(j,3),(n+j,3),(n+j,0))$ for $j\in \{2,4,\ldots,2n-2\}$). In particular, a maximal stable set in this induced subgraph has size $2(n-2)+2=2n-2$ and thus $|M|=2n$. This concludes the proof.
\end{proof}

\begin{proposition}\label{prop:S_n}
Let $n\geq 2$ and let $S_n$ be the Cayley graph on  $\ZZ_{2n}\times \ZZ_4$ with connection set $S=\{(0,1),(0,3),(2i+1,0),(2i+1,1),(2i+1,3)\mid 0\leq i \leq n-1\}$. Then $S_n$ is a connected CIS graph of order $8n$ and valency $3n+2$ with $\alpha(S_n)=2n$ and $\omega(S_n)=4$.
\end{proposition}

The proof of Proposition~\ref{prop:S_n} is very similar to the proof of Proposition~\ref{prop:R_n} and is omitted. Note that $R_2\cong Q_4\cong \overline{S_2}$. Using Gordon Royle's table of vertex-transitive graphs of order at most $32$ \cite{Royleweb}, we obtain the following with the help of computer.

\begin{proposition}\label{prop:up-to-32}
Let $\mathcal{F}$ be the family containing the following graphs:
\begin{enumerate}
\item $K_n$, $n\geq 1$,
\item $L(K_{n,n})$, $n\geq 3$,
\item $Q_n$, $n\geq 4$,
\item $R_n$, $n\geq 3$,
\item $S_n$, $n\geq 3$,
\end{enumerate}
and let $\overline{\mathcal{F}}$ be the closure of $\mathcal{F}$ under the operations of taking complements and lexicographic products. Then, up to isomorphism,  every vertex-transitive CIS graph of order at most $32$ is in $\overline{\mathcal{F}}$.
\end{proposition}

\section{Well-covered CIS graphs with clique number at most $3$}\label{sec:omega-3}

As we saw in Section~\ref{sec:prelim-CIS-well-covered}, when studying CIS graphs, we may restrict our attention to connected irreducible graphs. The main result of this section is Theorem~\ref{Omega=3}, a classification of connected irreducible well-covered CIS graphs with clique number at most $3$. First, we need a definition and a lemma.

\begin{definition}
Let $\Gamma$ be a graph and $v$ a non-isolated vertex in $\Gamma$. We denote by $\rho_\Gamma(v)$ the minimum of  $\{|S\cap N(v)|\mid S$ is a maximal stable set of $\Gamma$ not containing $v\}$ and by $\rho(\Gamma)$ the minimum of $\rho_\Gamma(v)$ as $v$ ranges over all non-isolated vertices of $\Gamma$.
\end{definition}

\begin{lemma}\label{lemma:intersection}
Let $\Gamma$ be a well-covered graph with no isolated vertex. Let $S$ be a maximal stable set in $\Gamma$, let $v\in V(\Gamma)\setminus S$ be such that $|S\cap N(v)|=\rho(\Gamma)$ and let  $X=\{v\}\cup (S\setminus N(v))$. Note that $X$ is a stable set. Let $S'$ be a maximal stable set containing $X$ and let $W=S'\setminus (S\cup\{v\})$.  Then $|S'\setminus S|=\rho(\Gamma)$, $|W|=\rho(\Gamma)-1$, $W\cap N[v]=\emptyset$ and every vertex of $W$ is  adjacent to every vertex of $S\cap N(v)$.
\end{lemma}

\begin{proof}
Note that $S\setminus S'=S\cap N(v)$ and, since every maximal stable set has cardinality $|S|$, $|S'\setminus S|=|S\setminus S'|=|S\cap N(v)|=\rho(\Gamma)$ and thus $|W|=\rho(\Gamma)-1$. Since $S'$ is a stable set, we have $W\cap N[v]=\emptyset$. Let $w\in W$. By definition of $\rho(\Gamma)$, $w$ has at least $\rho(\Gamma)$ neighbors in $S$, hence they must be exactly the elements of $S\setminus S'=S\cap N(v)$.
\end{proof}

\begin{theorem}\label{Omega=3}
Let $\Gamma$ be a connected irreducible well-covered CIS graph with $\omega(\Gamma)\leq 3$. Then $\Gamma$ is isomorphic to $K_1$, $K_2$, $K_3$ or $L(K_{3,3})$.
\end{theorem}
\begin{proof}
If $\omega(\Gamma)\leq 2$ then, by Corollary~\ref{Omega=2}, $\Gamma$ is isomorphic to $K_1$ or $K_2$. We thus assume that $\omega(\Gamma)=3$.

We first show that every edge of $\Gamma$ is contained in a triangle. Since $\omega(\Gamma)=3$ and $\Gamma$ is connected, it suffices to show that every edge that intersects a triangle is itself contained in a triangle. Let $T:=\{v,w,z\}$ be a triangle and let $x$ be a neighbor of $v$ not contained in $T$. Suppose that $\{x,v\}$ is not contained in a triangle. In particular, $x$ is adjacent to neither $w$ or $z$. Suppose that $x$ has a neighbor $y$ different than $v$. Since $\{x,v\}$ is not contained in a triangle, $y$ is not adjacent to $v$ and it follows by Observation~\ref{obs:bull} that neither $(y,x,v,w)$ nor $(y,x,v,z)$ is an induced path of length three. In particular, $y$ is adjacent to both $w$ and $z$. This implies that the edge $\{w,z\}$ is contained in two distinct triangles, contradicting Lemma~\ref{coollemma}. We may thus assume that $v$ is the unique neighbor of $x$. Let $S$ be a maximal stable set of $\Gamma$ containing $v$ and let $S' = (S\cup \{x\})\setminus \{v\}$. Clearly, $S'$ is a stable set of $\Gamma$ with $|S'| = |S|$ and, since $\Gamma$ is well-covered, it is a maximal stable set. On the other hand, $S'$ does not intersect the maximal clique $T$, contradicting the fact that $\Gamma$ is CIS. This concludes the proof that that every edge of $\Gamma$ is contained in a triangle and, in fact, in a unique triangle by Lemma~\ref{coollemma}. In particular, every vertex has even valency.

Let $v$ be a vertex of minimal valency in $\Gamma$ and let $r$ be the number of triangles containing $v$. Since every edge is contained in a unique triangle, $v$ has valency $2r$. Since $\Gamma$ is CIS, it is clear that $\rho_\Gamma(v)=r=\rho(\Gamma)$.

Let $u$ be a neighbor of $v$ and let $S$ be a maximal stable set of $\Gamma$ containing $u$ (and thus not $v$). Note that $|S\cap N(v)|=r$ and thus Lemma \ref{lemma:intersection} implies that there exists a stable set $S'$  containing $\{v\}\cup (S\setminus N(v))$ with $|S'\setminus S|=r$. Let $T$ be a triangle containing $u$. Since $\Gamma$ is CIS, $T\cap S'\neq\emptyset$. As $u\in S\setminus S'$, it follows that $(T\setminus\{u\})\cap (S'\setminus S)\neq \emptyset$. In particular, the number of triangles containing $u$ is at most $|S'\setminus S|=r$ and thus $u$ has valency $2r$. In particular, every neighbor of a vertex with minimal valency is also of minimal valency. Connectedness of $\Gamma$ implies that $\Gamma$ is $2r$-regular.

Suppose that $r\geq 3$ and let $N(v)=\{a_1,\ldots,a_r,b_1,\ldots,b_r\}$ such that $a_i$ is adjacent with $b_i$. Let $S$ be a maximal stable set in $\Gamma$ containing $\{a_1,\ldots,a_r\}$. By Lemma \ref{lemma:intersection}, there exists a set $W_1$ of $r-1$ vertices such that $W_1\cap N[v]=\emptyset$ but every vertex of $W_1$ is adjacent  to every vertex of $\{a_1,\ldots,a_r\}$. By the same argument, there exists a set $W_2$ of $r-1$ vertices such that $W_2\cap N[v]=\emptyset$ but every vertex of $W_2$ is adjacent to every vertex of $\{a_1,\ldots,a_{r-1},b_r\}$, and a set $W_3$ of $r-1$ vertices such that $W_3\cap N[v]=\emptyset$ but every vertex of $W_3$ is adjacent to every vertex of $\{a_1,\ldots, a_{r-2},b_{r-1},b_r\}$.

Note that $W_1\cap W_2=\emptyset$ otherwise the edge $\{a_r,b_r\}$ would be contained in two distinct triangles. Note also that every element of $(W_1\cup W_2)\cap W_3$ forms a triangle with $\{a_{r-1},b_{r-1}\}$. Since $v\notin(W_1\cup W_2)\cap W_3$ and every edge is in a unique triangle, it follows that $(W_1\cup W_2)\cap W_3=\emptyset$. This implies that $|W_1\cup W_2 \cup W_3|=3(r-1)$. Finally, note that $W_1,W_2,W_3,\{b_1,v\}\subseteq N(a_1)$, but neither $b_1$ nor $v$ are contained in $W_1\cup W_2 \cup W_3$. It follows that $a_1$ has valency at least $3(r-1)+2=3r-1$. This implies $k=2r\geq 3r-1$, contradicting the fact that $r\geq 3$.

We may thus assume that $r\leq 2$. If $r=1$ then clearly $\Gamma\cong K_3$. If $r=2$ then it follows from Theorem~\ref{Local TwiceComple} that $\Gamma\cong L(K_{3,3})$.
\end{proof}

Together with Theorem~\ref{thm:VT-CIS}, Theorem~\ref{Omega=3} immediately implies that a connected irreducible vertex-transitive CIS graph with clique number at most $3$ is isomorphic to $K_1$, $K_2$, $K_3$ or $L(K_{3,3})$. In contrast, if the clique number is $4$ then there are many infinite families of examples, for example $Q_n$, $R_n$ and $S_n$ (see Section~\ref{sec:examples}).

\section{Vertex-transitive CIS graphs of valency at most $7$}\label{sec:valency}

In this section, we classify vertex-transitive CIS graphs of valency at most $7$. We will need the following preliminary results.

\begin{lemma}\label{UniversalVertex}
Let $\Gamma$ be a $k$-regular graph such that every local graph of $\Gamma$ has exactly $n$ universal vertices. Then $\Gamma=Z[K_{n+1}]$ for some graph $Z$  of valency $\frac{k-n}{n+1}$.
\end{lemma}
\begin{proof}
Let $R$ be the equivalence relation ``having the same closed neighborhood'' on $V(\Gamma)$ and let $Z$ be the quotient graph of $\Gamma$ with respect to $R$. It is easy to see that $\Gamma=Z[K_{n+1}]$.
\end{proof}

\begin{lemma}\label{lemma:new}
Let $\Gamma$ be a connected, $k$-regular, well-covered, co-well-covered graph. Then either $\omega(\Gamma)\leq \frac{2}{3}(k+1)$ or $\Gamma$ is a complete graph.
\end{lemma}
\begin{proof}
Suppose that $\omega(\Gamma)> \frac{2}{3}(k+1)$. Let $\Gamma_{\mathcal{Q}}$ be the graph of maximal cliques of $\Gamma$, that is, the graph with maximal cliques of $\Gamma$ as vertices, and two such cliques adjacent in $\Gamma_{\mathcal{Q}}$ if they intersect in $\Gamma$. Since $\Gamma$ is co-well-covered and connected, it easily follows that $\Gamma_{\mathcal{Q}}$ is connected. By~\cite[Lemma~2.2]{CranstonRarXiv}, this implies that $\Gamma$ has a vertex that is contained in every maximal clique. Since $\Gamma$ is co-well-covered, this vertex is a universal vertex. In particular, there exists a maximal stable set of cardinality one. As $\Gamma$ is well-covered, every maximal stable set has cardinality one and thus $\Gamma$ is a complete graph.
\end{proof}

The graphs in Figure~\ref{fig:localgraphs} will play an important role in the proof of Theorem~\ref{thm:valency-7}.

\begin{figure}[h]
        \centering
        \begin{subfigure}[b]{0.45\textwidth}
        \centering
                \includegraphics[height=0.8in]{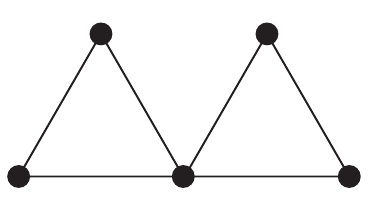}
                \caption{The graph $T_2$}
                \label{fig:T_2}
        \end{subfigure}
		\begin{subfigure}[b]{0.45\textwidth}
		\centering
                \includegraphics[height=0.8in]{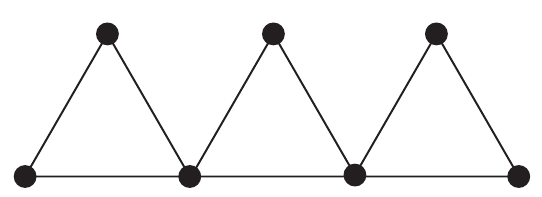}
                \caption{The graph $T_3$}
                \label{fig:T_3}
        \end{subfigure}

		\begin{subfigure}[b]{0.45\textwidth}
		\centering
                \includegraphics[height=1.2in]{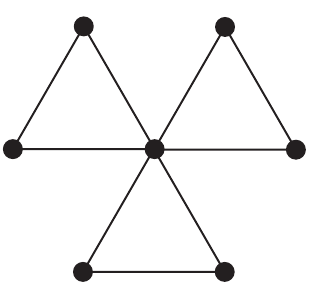}
                \caption{The graph $T_3'$}
                \label{fig:T_3'}
        \end{subfigure}
        \begin{subfigure}[b]{0.45\textwidth}
        \centering
                \includegraphics[height=1.1in]{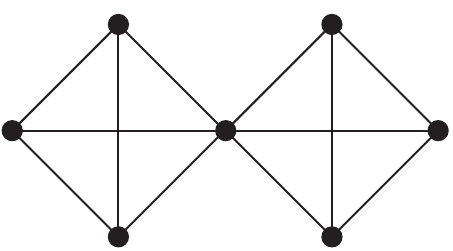}
                \caption{The graph $U_2$}
                \label{fig:U_2}
        \end{subfigure}
\caption{Local graphs for Theorem~\ref{thm:valency-7}}\label{fig:localgraphs}
\end{figure}

We leave the following lemma as an exercise for the reader.

\begin{lemma}\label{lemma:local}
Let $Y$ be a co-well-covered graph of order $k$ with $\omega(Y)=t$ and with the property that no two $t$-cliques intersect in a $(t-1)$-clique.
\begin{enumerate}
\item If $(k,t)=(5,3)$ then $Y\cong T_2$.
\item If $(k,t)=(6,3)$ then $Y\cong 2K_3$.
\item If $(k,t)=(7,3)$ then $Y\cong T_3$ or $T_3'$.
\item If $(k,t)=(7,4)$ then $Y\cong U_2$.
\end{enumerate}
\end{lemma}

\begin{theorem}\label{thm:valency-7}
Let $\Gamma$ be a connected, $k$-regular, irreducible, well-covered, and co-well-covered CIS graph. If $k\leq 7$ then $\Gamma$ is either a complete graph or isomorphic to one of $L(K_{3,3})$, $L(K_{4,4})$, $C_4[K_2]$, $K_{3,3}[K_2]$ or $Q_n$ for some $n\geq 4$.
\end{theorem}
\begin{proof}
By Theorem~\ref{Omega=3} and Lemma~\ref{lemma:new}, we may assume that $(k,\omega(\Gamma))$ is one of the following pairs : $(5,4)$, $(6,4)$, $(7,4)$, or $(7,5)$ (otherwise $\Gamma$ is either complete or isomorphic to $L(K_{3,3})$). Let $t=\omega(\Gamma)-1$ and let $Y$ be a local graph of $\Gamma$. By Corollary~\ref{corollary:localCIS}, $Y$ is a co-well-covered graph of order $k$ with $\omega(Y)=t$ and with the property that no two $t$-cliques intersect in a $(t-1)$-clique, and we may apply Lemma~\ref{lemma:local}.

If $k=5$ then every local graph is isomorphic to $T_2$. By Lemma~\ref{UniversalVertex}, $\Gamma=Z[K_2]$ for some $2$-regular graph $Z$. Since $\Gamma$ is not complete, $Z\cong C_n$ for some $n\geq 4$. By Proposition~\ref{lemma:CIS-complements}~(\ref{CIS:lexproduct}), $Z$ must be CIS hence $n=4$ and $\Gamma\cong C_4[K_2]$. If $k=6$ then every local graph is isomorphic to $2K_3$ and, by Theorem~\ref{Local TwiceComple}, $\Gamma\cong L(K_{4,4})$.

From now on, we assume that $k=7$. If $t=4$ then, by Lemma~\ref{lemma:local}, every local graph is isomorphic to $U_2$. Since $U_2$ contains a unique universal vertex, Lemma~\ref{UniversalVertex} implies that $\Gamma\cong Z[K_2]$ for some graph $Z$. It follows that $\omega(\Gamma) = 2\cdot\omega(Z)$ is even, contradicting the fact that $\omega(\Gamma) = t+1 = 5$.

From now on, we assume that $t=3$ and thus $\omega(\Gamma)=4$. By Lemma~\ref{lemma:local}, every local graph of $\Gamma$ is isomorphic to $T_3$ or $T_3'$. We will show that $\Gamma$ is either isomorphic to $K_{3,3}[K_2]$ or $Q_n$ for some $n\geq 4$. This case is by far the hardest; we thus break up the proof into a series of claims.

\smallskip

\noindent\textsc{Claim~1.}~Let $C$, $C_1$, and $C_2$ be three pairwise distinct $4$-cliques of $\Gamma$ such that $C_1\cap C_2=\emptyset$ and $C\subseteq C_1\cup C_2$. Then $|C\cap C_1|=|C\cap C_2|=2$ and every vertex of $C_1\setminus C$ is adjacent to every vertex of $C_2\setminus C$.

\smallskip

\noindent {\it Proof.} Recall that, by Lemma~\ref{coollemma}, distinct $4$-cliques of $\Gamma$ intersect in at most an edge and thus $|C\cap C_1|=|C\cap C_2|=2$. Let $x_1\in C_1\setminus C$ and let $x_2\in C_2\setminus C$. If $x_1$ and $x_2$ are not adjacent then they are part of a maximal stable set $S$. Since $C=(C_1\cap C)\cup(C_2\cap C)$, every vertex of $C$ is adjacent to one of $x_1$ or $x_2$. It follows that $S$ does not intersect $C$, contradicting the fact that $\Gamma$ is CIS.~$_\blacksquare$

\smallskip

We call an edge of $\Gamma$ \emph{red} if it is contained in at least two different $4$-cliques. If $\{u,v\}$ is a red edge, we will say that $u$ is a \emph{red} neighbor of $v$.

\smallskip

\noindent\textsc{Claim~2.}~Let $C$ be a $4$-clique containing two disjoint red edges $e_1$ and $e_2$. For every $i\in\{1,2\}$, let $C_i$ be a $4$-clique distinct from $C$ containing $e_i$. Then $C_1\cap C_2 = \emptyset$ and every vertex of $C_1\setminus C$ is adjacent to every vertex of $C_2\setminus C$. In particular, $C_i\setminus C$ is a red edge.

\medskip

\noindent {\it Proof.} Since distinct $4$-cliques intersect in at most an edge, we have $C_i\cap C=e_i$ and thus $|C_i\cap C|=2$ and $C=e_1\cup e_2\subseteq C_1\cup C_2$. Let $e_1=\{v,w\}$. In the local graph at $v$, $C$ and $C_1$ project to two triangles intersecting at the vertex $w$. In particular, there are no edges between $C_1\setminus C$ and $C\setminus C_1=C_2\cap C$. This implies that $C_1\cap C_2=\emptyset$. By Claim~1, every vertex of $C_1\setminus C$ is adjacent to every vertex of $C_2\setminus C$. It follows that $(C_1\setminus C)\cup (C_2\setminus C)$ is a $4$-clique and thus $C_i\setminus C$ is a red edge, being contained both in $C_i$ and in $(C_1\setminus C)\cup (C_2\setminus C)$.~$_\blacksquare$

\smallskip

\noindent\textsc{Claim~3.}~If  $\{u,v\}$ is an edge then the local graph at $u$ is isomorphic to the local graph at $v$.

\smallskip

\noindent {\it Proof.} Suppose otherwise, that is, without loss of generality, that the local graph at $u$ is isomorphic to $T_3$ and the local graph at $v$ is isomorphic to $T_3'$. Since $T_3'$ has a universal vertex, $v$ has a unique neighbor $v'$ such that $N[v]=N[v']$. Clearly, the local graph at $v'$ is also isomorphic to $T_3'$ and thus $v'\neq u$. Note that $v'$ is the unique red neighbor of $v$ and thus $\{u,v\}$ is not a red edge. By the same reasoning, $\{u,v'\}$ is not red either. Let $C$ be the unique $4$-clique containing $\{u,v\}$. Note that $v'\in C$ and write $C=\{u,v,v',w\}$. Since neither $v$ nor $v'$ are red neighbors of $u$ and the local graph at $u$ is isomorphic to $T_3$, it follows that $\{u,w\}$ is a red edge. Let $e_1=\{v,v'\}$ and $e_2=\{u,w\}$. There exist $C_1$ and $C_1'$, two $4$-cliques distinct from $C$ such that $C_1\cap C_1'=e_1$, and $C_2$ a $4$-clique distinct from $C$ and containing $e_2$. By Claim~2, every vertex of $C_2\setminus C$ is adjacent to every vertex of $(C_1\cup C_1')\setminus e_1$.

Let $z$ be the unique red neighbour of $u$ different from $w$. Note that $z\in C_2\setminus C$ and thus the four vertices of $(C_1\cup C_1')\setminus e_1$ are adjacent to $z$. On the other hand, by examining $T_3$, we see that $|N(z)\cap N(u)|=4$. Since the local graph at $v$ is isomorphic to $T_3'$, $u$ is not adjacent to any vertex in $(C_1\cup C_1')\setminus e_1$ and thus $N(u)$ is disjoint from $(C_1\cup C_1')\setminus e_1$. It follows that $|N(z)|\geq 8$, which is a contradiction.~$_\blacksquare$

\smallskip

Since $\Gamma$ is connected, Claim~3 implies that all local graphs of $\Gamma$ are pairwise isomorphic. The case when they are all isomorphic to $T_3'$ is easy to deal with. Indeed, this implies by Lemma~\ref{UniversalVertex} that $\Gamma=Z[K_2]$ for some connected $3$-regular graph $Z$ and it is easy to check that we must have $\omega(Z)=2$. By Proposition~\ref{lemma:CIS-complements}~(\ref{CIS:lexproduct}), $Z$ must be a CIS graph. By Corollary~\ref{YOYO}, $Z\cong K_{3,3}$ and consequently $\Gamma\cong K_{3,3}[K_2]$.

From now on, we assume that all local graphs of $\Gamma$ are isomorphic to $T_3$. In particular, every red edge is part of exactly two $4$-cliques, and every vertex has precisely two red neighbors (the vertices of valency $4$ in its local graph) and thus is incident to exactly two red edges. Let $v$ be a vertex. Starting from the local graph at $v$ and removing the two red neighbors of $v$, we obtain a graph with five vertices and two edges. We call these two edges \emph{extremal (with respect to $v$)}.

\smallskip

\noindent\textsc{Claim~4.}~Every extremal edge is red.

\smallskip

\noindent {\it Proof.} Note that each red edge appears in exactly four local graphs. On the other hand, in the local graph at a vertex, red neighbors are incident to at most one red edge and thus the local graph contains at most four red edges. Since the number of red edges is equal to the number of vertices, every local graph contains exactly four red edges, two of which must include the extremal edges.~$_\blacksquare$

\smallskip

\noindent\textsc{Claim~5.}~Every red $2$-path is contained in a unique red cycle, which has length $4$. Moreover, the induced graph on this $4$-cycle is a clique.

\smallskip

\noindent {\it Proof.} Let $P=(u,v,w)$ be a red $2$-path. Clearly, $u$ and $w$ are the only two red neighbors of $v$. By considering the local graph at $v$, it follows that $u$ and $w$ are adjacent and there is a unique vertex, say $x$, adjacent to each of $u$, $v$ and $w$. Now, $\{u,v,w\}$ is a triangle in the local graph at $x$ therefore it contains a red neighbor of $x$. Since $x\notin V(P)$ and $u$ and $w$ are the only two red neighbors of $v$, we may assume without loss of generality that $w$ is a red neighbor of $x$.

Let $e_1=\{v,w\}$ and let $C_1=\{u,v,w,x\}$. Note that $C_1$ is a $4$-clique. Let $C$ be the unique $4$-clique distinct from $C_1$ and containing the edge $e_1$. Note that $C\setminus C_1$ is an extremal edge with respect to $v$ hence it must be red by Claim~4.  By Claim~2, it follows that $C_1\setminus C=\{u,x\}$ must be red, concluding the proof of this claim.~$_\blacksquare$

\smallskip

We call a $4$-clique \emph{red} if it contains a red cycle, and \emph{black} otherwise.

\smallskip

\noindent\textsc{Claim~6.}~Every black clique contains exactly two red edges and these are disjoint.

\smallskip

\noindent {\it Proof.}
Let $C$ be a black clique. By Claim~5, $C$ does not contain a red $2$-path and thus it suffices to show that $C$ contains at least two red edges. Let $v\in C$. Suppose first that $C$ contains no edge that is extremal with respect to $v$. By examining $T_3$, we see that this implies that $C$ contains two red neighbours of $v$ and thus a red $2$-path, which is a contradiction. We may thus assume that $C$ contains an extremal edge with respect to $v$, say $e$. By examining $T_3$, we see that $C=e\cup\{v,w\}$ for some red neighbor $w$ of $v$. By Claim~4, $e$ is red and thus $C$ contains at least two red edges.~$_\blacksquare$

\smallskip

Let $C$ be a red clique and let $e_1$, $e_2$, $e_3$, and $e_4$ be the four red edges of $C$, labeled such that $e_1$ and $e_2$ are disjoint (and thus also $e_3$ and $e_4$). For every $i\in\{1,2,3,4\}$, there is a unique $4$-clique containing $e_i$ and distinct from  $C$, say $C_i$. Clearly, $C_i$ is black and thus, by Claim~$6$, it contains a unique red edge distinct from $e_i$, say $f_i$. Note that $e_i\cap f_i=\emptyset$. Now, $f_i$ is red and thus it is contained in a red clique, say $R_i$. Since every red edge is contained in exactly two $4$ cliques, it follows that $R_i$ is the unique clique distinct from $C_i$ and containing $f_i$.  By Claim~2, every vertex of $f_1$ is adjacent to every vertex of $f_2$. In particular, $f_1\cup f_2$ is a $4$-clique and thus $R_1=R_2$. By the same reasoning, we have $R_3=R_4$. Note that $f_1$ and $f_3$ are extremal edges of the local graph at the vertex $e_1\cap e_3$. In particular, $f_1$ and $f_3$ are disjoint and there are no edges between them. It follows that $R_3\neq R_1$.

Note that $R_i$ was uniquely determined by $C$ and $e_i$. We call $R_i$ the red clique \emph{adjacent to $C$ at $e_i$}. Note that, in fact, $C$ is adjacent to $R_i$ at $f_i$, hence the adjacency relation is symmetric. Moreover, we have shown in the paragraph above that every red clique is adjacent to exactly two red cliques.

In particular, there is a set of $n$ red cliques indexed by $\ZZ_n$ with $n\geq 3$ such that $C_i$ is adjacent to $C_{i+1}$. It is not hard to check that $\Gamma$ being connected implies that $\bigcup_{i\in\ZZ_n} C_i=V(\Gamma)$.

Now, for every $i\in \ZZ_n$, there are exactly two edges, say $e_{i,0}$ and $e_{i,1}$ of $C_i$ such that $C_{i-1}$ is adjacent to $C_i$ at $e_{i,0}$ and also at $e_{i,1}$. Note that $C_i=e_{i,0}\cup e_{i,1}$.

We define a mapping $\varphi:\bigcup_{i\in\ZZ_n} C_i\to \ZZ_n\times\ZZ_2\times\ZZ_2$ by $\varphi(v)=(i,x,y)$ such that $v$ is contained in $e_{i,x}$ and adjacent to $e_{i+1,y}$. Now we will show that $\varphi$ is a well-defined bijection. Indeed, if $v\in \bigcup_{i\in\ZZ_n} C_i$ then there is a unique $i\in \ZZ_n$ such that $v\in C_i$. Since $C_i$ is the disjoint union of $e_{i,0}$ and $e_{i,1}$, there is a unique $x\in \ZZ_2$ such that $v\in e_{i,x}$. Let $f$ be the unique red edge containing $v$ distinct from $e_{i,x}$. Note that the red clique adjacent to $C_i$ at $e_{i,x}$ is $C_{i-1}$ and thus the red clique adjacent to $C_i$ at $f$ is $C_{i+1}$. It follows that there is a unique $y\in \ZZ_2$ such that $v$ is adjacent to $e_{i+1,y}$.

We now show that $\varphi$ is an isomorphism  between $\Gamma$ and $Q_n$. Let $\{v,w\}\in E(\Gamma)$ and let $\varphi(v)=(i,x,y)$. Note that $N[v] \subseteq C_{i-1}\cup C_i \cup C_{i+1}$. If $w\in C_i$ then $\varphi(v)$ and $\varphi(w)$ have the same first coordinate, hence $\varphi(\{v,w\})\in E(Q_n)$. Suppose now that $w\not\in C_{i}$. Without loss of generality, we may assume that $w\in C_{i+1}$. Note that $v\in e_{i,x}$ and $w\in e_{i+1,y}$ and thus $\varphi(w)=(i+1,y,z)$ for some $z\in \ZZ_2$. This shows that $\varphi(\{v,w\})\in E(Q_n)$. As $\{v,w\}$ was an arbitrary edge of $\Gamma$, it follows that $\varphi(E(\Gamma))=E(Q_n)$. Since $|E(\Gamma)|=|E(Q_n)|$, this implies that $\varphi$ is an isomorphism. It is not hard to check that $Q_3$ is not CIS and thus $n\geq 4$.
\end{proof}

\begin{proposition}\label{irreducibleQuotientVT}
Let $\Gamma$ be a graph and let $X$ be its irreducible quotient. Then, $\Gamma$ is vertex-transitive if and only if $X$ is vertex-transitive and $\Gamma \cong X[\overline{K_n}]$ for some $n\ge 1$.
\end{proposition}

\begin{proof}
 Let $R$ be the equivalence relation ``having the same neighborhood'' on $V(\Gamma)$. If $\Gamma$ is vertex-transitive then the $R$-equivalence classes have the same size, say $n$. Since these equivalence classes are stable sets, it follows that $\Gamma \cong X[\overline{K_n}]$. Moreover, the $R$-equivalence classes form a partition of $V(\Gamma)$ invariant under the action of the automorphism group of $\Gamma$ and thus the automorphism group of $\Gamma$ has an induced action on the quotient, namely $X$. Since $\Gamma$ is vertex-transitive, it follows that $X$ is vertex-transitive. The converse is clear.
\end{proof}

Combining Propositions~\ref{irreducibleQuotient} and~\ref{irreducibleQuotientVT} and Theorems~\ref{thm:VT-CIS} and~\ref{thm:valency-7} we easily obtain the following.

\begin{corollary}\label{cor:VT-CIS-valency-7}
Let $\Gamma$ be a connected vertex-transitive graph of valency at most $7$. Then, $\Gamma$ is CIS if and only if $\Gamma$ is isomorphic to one of the following graphs $K_n$ ($1\leq n \leq 8$), $K_{n,n}$ ($2\leq n\leq 7$), $L(K_{3,3})$, $L(K_{4,4})$, $C_4[K_2]$, $K_3[\overline{K_2}]$, $K_3[\overline{K_3}]$, $K_4[\overline{K_2}]$, $K_{3,3}[K_2]$ or $Q_n$ for some $n\geq 4$.
\end{corollary}

\section{Some open questions}\label{sec:questions}

We conclude the paper with some open questions related to the results of this paper. It is known that every perfect graph $\Gamma$ satisfies the inequality $\alpha(\Gamma)\omega(\Gamma)\ge |V(\Gamma)|$~\cite{MR0309780}, and Theorem~\ref{thm:VT-CIS} implies that it holds with equality for vertex-transitive CIS graphs. This motivates the following question.

\begin{question}
Does every CIS graph $\Gamma$ satisfy $\alpha(\Gamma)\omega(\Gamma)\ge |V(\Gamma)|$?
\end{question}

Given the examples that have appeared in this paper, the following is also a natural question.

\begin{question}\label{quest-chi}
Does every vertex-transitive CIS graph $\Gamma$ admit a decomposition of its vertex set into $\omega(\Gamma)$ stable sets?
\end{question}

Question~\ref{quest-chi} can be stated equivalently as follows: {\it Does the chromatic number of every vertex-transitive CIS graph equal its clique number?} Generalizing the famous class of perfect graphs~\cite{MR2233847,MR2063679}, graphs with this property were named {\it weakly perfect graphs} in~\cite{MR2679902,MR2738964}. Finally, in view of Theorem~\ref{thm:valency-7}, the following question seems natural.

\begin{question}
{Let $G$ be a regular CIS graph such that both it and its complement are connected, well-covered and irreducible. Is $G$ necessarily  vertex-transitive?}
\end{question}

We suspect the answer is no, but it does not seem easy to produce a counterexample.

\bibliographystyle{abbrv}

\begin{sloppypar}
\bibliography{CIS-VT-bib}{}

\end{sloppypar}

\end{document}